\documentclass{amsart}
\usepackage[utf8]{inputenc}
\usepackage{amsfonts}
\usepackage{mathtools}
\usepackage{amsthm}
\usepackage[hidelinks]{hyperref}
\usepackage{IEEEtrantools}
\usepackage{amsmath}
\usepackage{amssymb}
\usepackage{enumerate}
\usepackage{physics}
\usepackage{xcolor}
\usepackage{microtype}

\newcommand{\B}{\mathcal{B}}
\newcommand{\E}{\mathbb{E}}
\newcommand{\C}{\mathbb{C}}
\newcommand{\D}{\mathbb{D}}
\newcommand{\Dbar}{\overline{\D}}
\newcommand{\R}{\mathbb{R}}
\newcommand{\T}{\mathbb{T}}
\newcommand{\Z}{\mathbb{Z}}
\newcommand{\N}{\mathbb{N}}
\newcommand{\loc}{\mathrm{loc}}
\renewcommand{\epsilon}{\varepsilon}
\DeclareMathOperator{\supp}{supp}

\let\temp\phi
\let\phi\varphi
\let\varphi\temp

\DeclarePairedDelimiter\absNew{\lvert}{\rvert}
\DeclarePairedDelimiter\normNew{\lVert}{\rVert}

\renewcommand{\abs}{\absNew}
\renewcommand{\norm}{\normNew}

\numberwithin{equation}{section}

\newtheorem{theorem}{Theorem}[section]
\newtheorem{proposition}[theorem]{Proposition}
\newtheorem{lemma}[theorem]{Lemma}
\newtheorem{corollary}[theorem]{Corollary}

\theoremstyle{definition}
\newtheorem{definition}[theorem]{Definition}

\title
[Tauberian theorems for sequences]
{Tauberian theorems for sequences and the Katznelson--Tzafriri theorem}

\begin{document}

\author[A.K.J. Pritchard and D. Seifert]{Andrew K. J. Pritchard and David Seifert}
\address{School of Mathematics, Statistics and Physics, Newcastle University, Herschel Building, Newcastle upon Tyne, NE1 7RU, United Kingdom}
\email[A.K.J.\ Pritchard]{a.k.j.pritchard2@ncl.ac.uk}
\email[D.\ Seifert]{david.seifert@ncl.ac.uk}



\begin{abstract}
In this note, we present an alternative proof of a quantified Tauberian theorem for vector-valued sequences first proved in \cite{Sei15_Tauberian}. The theorem relates the decay rate of a bounded sequence with properties of a certain boundary function. We present a slightly strengthened version of this result, and illustrate how it can be used to obtain quantified versions of the  Katznelson--Tzafriri  theorem as well as results on Ritt operators.
\end{abstract}

\maketitle

\section{Introduction}

The classical Katznelson--Tzafriri theorem \cite[Theorem~1]{KatTza86_Power_bounded} states that, given a power-bounded operator $T \in \B(X)$ on a complex Banach space $X$, we have ${\lim_{n\to\infty} \norm{T^n(I-T)} = 0}$ if and only if $\sigma(T) \cap \T \subseteq \{1\}$.  Here we let  $\T \coloneqq \{\lambda \in \C : \abs{\lambda} = 1\}$, $\B(X)$ denotes the space of bounded linear operators on $X$, and $\sigma(T)$ is the spectrum of $T$.

In view of the many applications of results of this type (see e.g.\ \cite{Nev93_Iterations_book}, or the survey \cite[Section~4]{BatSei22_KT_Developments} and references therein), it is of interest to investigate the \emph{rate} at which the norms $\norm{T^n(I-T)}$ decay as $n \to \infty$. Such a quantified Katznelson--Tzafriri theorem was obtained by the second author in \cite{Sei15_Tauberian} and \cite{Sei16_KT_Rates}, through two different approaches. In \cite{Sei16_KT_Rates}, an approach based on a functional calculus argument was used, whereas the approach taken in \cite{Sei15_Tauberian} was to prove a quantified Tauberian theorem for bounded  sequences taking values in a Banach space.

The Tauberian theorem shows that if $x \in \ell^\infty(\Z_+;X)$ has bounded partial sums and a well-defined boundary function $F_x \in L^1_{\loc}(\T\setminus\{1\};X)$ (see Definition~\ref{def:bdy fctn} below), then $x \in c_0(\Z_+;X)$. It further quantifies the decay rate when $F_x \in C^k(\T \setminus\{1\};X)$ for $k \in \N \cup \{\infty\}$ and we have bounds on the derivatives of $F_x$. 

The first purpose of this paper is to give a modified proof of the quantified Tauberian theorem \cite[Theorem~2.1]{Sei15_Tauberian}. Specifically, in the proof of the Tauberian theorem in \cite{Sei15_Tauberian}, the sequence $x \in \ell^\infty(\Z_+;X)$ is approximated by another sequence $x^\epsilon$, which is defined as the convolution of $x$ with a sequence of Fourier coefficients of a function $\theta \mapsto \phi(\theta/\epsilon)$. The function $\phi$ chosen in \cite{Sei15_Tauberian} is a specific piecewise-linear function, and the proof of the theorem involves estimating particular boundary terms. The alternative proof we give below shows that we may replace the specific function $\phi$ with a sequence of sufficiently differentiable functions $(\phi_\ell)_{\ell=1}^\infty$, thus eliminating the need for a specific choice of $\phi$ and the computation of boundary terms. In fact, our result provides a slight improvement on the rate obtained previously.

The other purposes of this paper are to give a straightforward extension of the quantified Tauberian theorem to sequences with boundary functions $F_x $ defined on $\T\setminus E$, where $E $ is a finite subset of $\T$, and to show how the same techniques used in the proof of the Tauberian theorems can be applied to recover some known results about Ritt operators.

In Section~\ref{sec:single point}, we give a modified proof of Theorem~\ref{thm:disct taub single} using techniques from \cite[Section~1.3]{Hor03_Volume_I} and use this approach to recover a quantified version of the Katznelson--Tzafriri theorem and some facts about Ritt operators. 

In Section~\ref{sec:several points}, we state a natural extension of Theorem~\ref{thm:disct taub single} to sequences whose boundary function  has several singularities on the unit circle, and we apply this to deduce a generalised version of the quantified Katznelson--Tzafriri theorem.

We use the following notation throughout. We let $\Dbar \coloneqq \{\lambda \in \C : \abs{\lambda} \leq 1\}$ denote the closed unit disc in $\C$, and set $\E \coloneqq \C \setminus \Dbar$. Given a complex Banach space $X$ and an operator $T \in \B(X)$, we let  $R(\lambda,T) \coloneqq (\lambda I - T)^{-1}$ for any $\lambda \in \C \setminus \sigma(T)$. If $p,q$ are real-valued, then we write $p \lesssim q$ when there exists $C > 0$ such that $p \leq C q$. The constant $C$ is assumed to be independent of any parameters which are free to vary in the given context. Given a complex-valued function $\phi$, we denote by $\supp \phi$ the (closed) support of $\phi$.

\subsection*{Acknowledgements}
The authors would like to thank Gregory Debruyne for useful discussions and suggestions regarding the proof of Theorem~\ref{thm:disct taub mod}.


\section{Tauberian Theorems with Single Boundary Singularity}\label{sec:single point}

We begin with the definition of a \emph{boundary function}.
\begin{definition}\label{def:bdy fctn}
Let $X$ be a complex Banach space and let $G\colon \E \to X$ be a holomorphic function. A function $F \in L^1_{\loc}(\T\setminus\{1\};X)$ is a \emph{boundary function} for $G$ if
\begin{equation*}
    \lim_{r \to 1+} \int_{-\pi}^\pi \psi(\theta) G(re^{i\theta}) \dd\theta
    = \int_{-\pi}^\pi \psi(\theta) F(e^{i\theta}) \dd\theta
\end{equation*}
for all $\psi \in  C([-\pi,\pi])$ such that $ 0 \notin \supp \psi$.
\end{definition}

If $X$ is a complex Banach space and $x \in \ell^\infty(\Z_+;X)$, then we may define a holomorphic function $G_x\colon \E \to X$ by
\begin{equation*}
    G_x(\lambda) \coloneqq \sum_{n = 0}^\infty \frac{x_n}{\lambda^{n+1}}.
\end{equation*}
Given $\delta > 0$, $k \in \N$ and a continuous, non-increasing function $m\colon (0,\delta] \to (0,\infty)$, we define functions $m_k, m_{\log} \colon (0,\delta] \to (0,\infty)$ by
\begin{equation*}
    m_k(\epsilon) \coloneqq m(\epsilon) \left( \frac{m(\epsilon)}{\epsilon} \right)^{1/k},
    \qquad
    m_{\log}(\epsilon) \coloneqq m(\epsilon) \log\left( 1 + \frac{m(\epsilon)}{\epsilon} \right).
\end{equation*}

\begin{theorem}[{\cite[Theorem~2.1]{Sei15_Tauberian}}]\label{thm:disct taub single}
Let $X$ be a complex Banach space and suppose that $x \in \ell^\infty(\Z_+;X)$ satisfies
\begin{equation}\label{eq:partial sum bound}
    \sup_{n\ge0} \bigg\|{\sum_{k=0}^n x_k} \bigg\|< \infty.
\end{equation}
If there exists a boundary function $F_x \in L^1_{\loc}(\T \setminus \{1\}; X)$ for $G_x$, then $x \in c_0(\Z_+;X)$. 

Furthermore, if $m\colon (0,\pi] \to (0,\infty)$ is a continuous, non-increasing function, then the following hold:
\begin{enumerate}[(i)]
    \item[\textup{(i)}]  Suppose that $F_x \in C^k(\T \setminus \{1\}; X)$ for some $k \geq 1$ and that there exists  $C>0$ such that 
    \begin{equation*}
        \norm{F_x^{(j)}(e^{i\theta})} \leq C \abs{\theta}^{\ell - j} m(\abs{\theta})^{\ell+1},\qquad \abs{\theta} \in (0,\pi],
    \end{equation*}
for   $0 \leq j \leq \ell \leq k$. Then, for every $c > 0$,
    \begin{equation*}
        \norm{x_n} = O\big(m_k^{-1}(cn)\big), \qquad n \to \infty.
    \end{equation*}

    \item[\textup{(ii)}]  \label{thm:disct taub smooth}
    Suppose that $F_x \in C^\infty(\T \setminus \{1\}; X)$ and that there exists  $C>0$ such that 
    \begin{equation*}
        \norm{F_x^{(j)}(e^{i\theta})} \leq C j! \abs{\theta} m(\abs{\theta})^{j+1},\qquad \abs{\theta} \in (0,\pi],
    \end{equation*}
 for all $j \geq 0$. Then, for every $c \in(0,1)$,
    \begin{equation*}
        \norm{x_n} = O\left( \frac{1}{n} + m_{\log}^{-1}(cn) \right), \qquad n \to \infty.
    \end{equation*}
\end{enumerate}
\end{theorem}

In order to provide our alternative proof, we require the following lemma based on techniques in \cite[Section~1.3]{Hor03_Volume_I}, which allows us to construct functions $\phi_\ell$ with an arbitrary number of derivatives and uniform upper bounds for those derivatives.

\begin{lemma}\label{lem:step funcs conv}
Let $k_0 \in \N$. For every $\ell \in \N$, there exists a function $\phi_\ell \in C^{\ell+k_0}(\R)$ and a constant $C > 1$ independent of $\ell$ such that the following hold:
\begin{enumerate}
    \item[\textup{(i)}] $\phi_\ell(\theta) = 1$ for $\abs{\theta} \leq 1$;
    
    \item[\textup{(ii)}]  $\phi_\ell(\theta) = 0$ for $\abs{\theta} \geq 2$;
    
    \item[\textup{(iii)}]  $\phi_\ell(\theta) \in (0,1)$ for  $1 < \abs{\theta} < 2$;
    
    \item[\textup{(iv)}]  $\max_{1 \leq j \leq k_0+1} \norm{\phi_\ell^{(j)}}_{L^\infty(\R)} \leq C$;
    
    \item[\textup{(v)}]  $\norm{\phi_\ell^{(j)}}_{L^\infty(\R)} \leq C (2^{k_0+2} \ell)^{j-(k_0+1)}$ for $k_0+1 < j \leq \ell + k_0$.
\end{enumerate}
\end{lemma}

\begin{proof}
For real numbers $a,b$ with $a < b$, define the step function $H_{[a,b]}\colon \R \to \R$ by
\begin{equation*}
    H_{[a,b]}(\theta) \coloneqq (b-a)^{-1} \chi_{[a,b]}(\theta),
\end{equation*}
where $\chi_{[a,b]}$ is the characteristic function of the interval $[a,b]$. If $\phi \in C^m(\R)$ and $a>0$, then
\begin{equation*}
    \phi* H_{[-a,a]}(\theta) 
    = \frac{1}{2a} \int_{-a}^a \phi(\theta-s) \dd{s}
    = \frac{1}{2a} \int_{\theta-a}^{\theta+a} \phi(s) \dd{s}.
\end{equation*}
Hence $\phi* H_{[-a,a]} \in C^{m+1}(\R)$ with 
\begin{equation*}
(\phi * H_{[-a,a]})'(\theta) = \frac{\phi(\theta+a) - \phi(\theta-a)}{2a},    
\end{equation*}
which implies that
\begin{equation*}
    \norm{(\phi * H_{[-a,a]})'}_{L^\infty(\R)} \leq \frac{1}{a} \norm{\phi}_{L^\infty(\R)}.
\end{equation*}

If $\phi \in L^1(\R)$ and $0 < a < b$ are such that $\phi(\theta) = 1$ for $\abs{\theta} \leq a$ and $\phi(\theta) = 0$ for $\abs{\theta} \geq b$, then for all $c \in (0,a)$, we have $(\phi * H_{[-c,c]})(\theta) = 1$ for $\abs{\theta} \leq a - c$ and $(\phi * H_{[-c,c]})(\theta) = 0$ for $\abs{\theta} \leq b + c$.

We now define $\phi_{\ell}$ as
\begin{equation*}
    \phi_\ell \coloneqq \chi_{[-3/2,3/2]} * H_{[-\frac{1}{4}, \frac{1}{4}]} * H_{[-\frac{1}{8}, \frac{1}{8}]}
    * \ldots * H_{[-\frac{1}{2^{k_0+2}}, \frac{1}{2^{k_0+2}}]} * (H_{[-\frac{1}{2^{k_0+2}\ell}, \frac{1}{2^{k_0+2}\ell}]})^{*\ell}.
\end{equation*}
By the above calculations, we see that $\phi_\ell$ satisfies all of the claimed properties.
\end{proof}

We are now able to give an alternative proof of Theorem~\ref{thm:disct taub single}(ii). Note also that we obtain a slight improvement in the rate of decay.

\begin{theorem}\label{thm:disct taub mod}
Let $X$ be a complex Banach space and suppose that $x \in \ell^\infty(\Z_+;X)$ satisfies~\eqref{eq:partial sum bound}.
Suppose further that $G_x$ admits a smooth boundary function $F_x \in C^\infty(\T\setminus\{1\}; X)$. Let $m\colon (0,\pi] \to (0,\infty)$ be a continuous, non-increasing function and assume there exists $C > 0$ such that
\begin{equation}\label{eq:bdy func deriv bound}
    \norm{F_x^{(j)}(e^{i\theta})} \leq C j! \abs{\theta} m(\abs{\theta})^{j+1}, \qquad 0 < \abs{\theta} \leq \pi,
\end{equation}
for all $j\ge0$.
Then, for every $c \in (0,1)$,
\begin{equation}\label{eq:mlog seq bound}
    \norm{x_n} = O\big(m_{\log}^{-1}(cn)\big), \qquad n \to \infty.
\end{equation}
\end{theorem}

\begin{proof}
We use a similar approach to the proof of \cite[Theorem~2.1]{Sei15_Tauberian}. For each $\ell \geq 1$, let $\phi_\ell$ be a function as in Lemma~\ref{lem:step funcs conv} for $k_0 = 1$, restricted to the interval $[-\pi,\pi]$. Let $\psi_\ell(\theta) \coloneqq 1 - \phi_\ell(\theta)$ and, for $\epsilon \in (0,\pi/2]$, we define sequences $y^{\epsilon,\ell}$, $z^{\epsilon,\ell}$ by
\begin{equation*}
    y_n^{\epsilon,\ell} \coloneqq \frac{1}{2\pi} \int_{-\pi}^\pi e^{in\theta} \psi_\ell(\theta/\epsilon) \dd{\theta}, \qquad
    z_n^{\epsilon,\ell} \coloneqq \frac{1}{2\pi} \int_{-\pi}^\pi e^{in\theta} \phi_\ell(\theta/\epsilon) \dd{\theta}
\end{equation*}
for each $n \in \Z$. Note that as $\phi_\ell, \psi_\ell \in C^{\ell+1}([-\pi,\pi]) \subseteq C^2([-\pi,\pi])$, integrating by parts twice gives $y^{\epsilon,\ell}, z^{\epsilon,\ell} \in \ell^1(\Z)$. Hence, after extending $x$ to a sequence in $\ell^\infty(\Z;X)$ by setting $x_n = 0$ for $n < 0$, we may define the convolution $x^{\epsilon,\ell} \coloneqq x * y^{\epsilon,\ell} \in \ell^\infty(\Z;X)$. We now estimate separately $\norm{x_n - x^{\epsilon,\ell}_n}$ and $\norm{x^{\epsilon,\ell}_n}$ for $n \geq 0$.

Note that for each $n \in \Z$,
\begin{equation*}
    x_n - x^{\epsilon,\ell}_n = (x * z^{\epsilon,\ell})_n = \sum_{j=0}^\infty s_j (z^{\epsilon,\ell}_{n-j} - z^{\epsilon,\ell}_{n-j-1}),
\end{equation*}
where $s_n \coloneqq \sum_{j=0}^n x_j$. Also, for each $n \in \Z$,
\begin{equation*}
    z^{\epsilon,\ell}_n - z^{\epsilon,\ell}_{n-1} = \frac{1}{2\pi} \int_{-\pi}^\pi e^{in\theta} (1 - e^{-i\theta}) \phi_\ell(\theta/\epsilon) \dd{\theta}.
\end{equation*}
Using the fact that $\supp \phi_{\ell} \subseteq [-2,2]$ and $\norm{\phi_\ell}_{L^\infty(\R)} = 1$ we have, for every $n \in\Z$,
\begin{equation}
    \abs{z^{\epsilon,\ell}_n - z^{\epsilon,\ell}_{n-1}} 
     \leq 
    \frac{1}{2\pi} \int_{-\pi}^\pi \abs{1 - e^{-i\theta}} \abs{\phi_\ell(\theta/\epsilon)} \dd{\theta} 
 \lesssim
    \int_{-2\epsilon}^{2\epsilon} \abs{\theta} \dd{\theta} 
    \lesssim
    \epsilon^2,
    \label{eq:z-eps bound}
\end{equation}
where here, and for the remainder of the proof, the implicit constant is independent of $\epsilon$, $\ell$ and $n$. If $n \neq 0$, then integrating by parts twice gives
\begin{equation}
\begin{aligned}
    \abs{z^{\epsilon,\ell}_n - z^{\epsilon,\ell}_{n-1}} 
    & = 
    \abs*{\frac{1}{2n^2\pi} \int_{-\pi}^\pi e^{in\theta} \dv[2]{\theta} ((1 - e^{-i\theta}) \phi_\ell(\theta/\epsilon)) \dd{\theta}} 
    \\
    & \lesssim
    \frac{1}{n^2} \int_{-\pi}^\pi \abs{\phi_\ell(\theta/\epsilon)} + \frac{\abs{\phi_\ell'(\theta/\epsilon)}}{\epsilon} + \frac{\abs{(1 - e^{i\theta}) \phi_\ell''(\theta/\epsilon)}}{\epsilon^2} \dd{\theta}
    \lesssim
    \frac{1}{n^2},
    \label{eq:z-n_square bound}
\end{aligned}
\end{equation}
where the implicit constant depends only on $\max_{j=0,1,2} \norm{\phi_\ell^{(j)}}_{L^\infty([-\pi,\pi])}$, which we have taken to be independent of $\ell$ by Lemma~\ref{lem:step funcs conv}. Now, using \eqref{eq:partial sum bound}, \eqref{eq:z-eps bound} and \eqref{eq:z-n_square bound}, we have 
\begin{equation}\label{eq:x_n-x_ne bound}
    \norm{x_n - x^{\epsilon,\ell}_n}
    \lesssim
    \sum_{\substack{j \geq 0 \\ \abs{j-n} \leq 1/\epsilon}} \hspace{-1em} \epsilon^2 \hspace{1em}
    + \sum_{\substack{j \geq 0 \\ \abs{j-n} > 1/\epsilon}} \hspace{-1em} \frac{1}{(n-j)^2}
    \lesssim \epsilon
\end{equation}
for all $n\ge0$.

We now estimate $\norm{x^{\epsilon,\ell}_n}$. As in \cite{Sei15_Tauberian}, we may use the dominated convergence theorem, Fubini's theorem and the definition of a boundary function to deduce that, for every $n \geq 0$,
\begin{equation*}
    x^{\epsilon,\ell}_n = \frac{1}{2\pi} \int_{-\pi}^\pi e^{i(n+1) \theta} F_x(e^{i\theta}) \psi_\ell(\theta/\epsilon) \dd{\theta}.
\end{equation*}
Integrating by parts $k \geq 1$ times gives
\begin{equation*}
    x^{\epsilon,\ell}_n = 
    \frac{1}{2\pi} \sum_{j=0}^k \frac{A_{n,k,j}}{\epsilon^j} \int_{-\pi}^\pi e^{i(n+k-j+1) \theta} F_x^{(k-j)}(e^{i\theta}) \psi_\ell^{(j)}(\theta/\epsilon) \dd{\theta},
\end{equation*}
where the complex-valued coefficients $A_{n,k,j}$ satisfy
\begin{equation*}
    \abs{A_{n,k,j}} \leq \frac{1}{n^k}\binom{k}{j} .
\end{equation*}
Using this with Lemma~\ref{lem:step funcs conv} and \eqref{eq:bdy func deriv bound}, and noting that $\psi_\ell^{(j)} = -\phi_\ell^{(j)}$ for $j \geq 1$, we have
\begin{align*}
    \norm{x^{\epsilon,\ell}_n} 
    & \lesssim 
    \sum_{j=0}^k \binom{k}{j} \frac{1}{n^k\epsilon^j} \int_{-\pi}^\pi \norm{F_x^{(k-j)}(e^{i\theta})} \, \abs{\psi_\ell^{(j)}(\theta/\epsilon)} \dd{\theta} 
    \\[2pt]
    & \lesssim 
    \frac{k!}{n^k} \int_{\abs{\theta}>\epsilon} \abs{\theta} m(\abs{\theta})^{k+1} \dd{\theta} 
    \\
    & \qquad +
    \sum_{j=1}^k \frac{k!}{j!} \frac{1}{n^k\epsilon^j} \norm{\psi_\ell^{(j)}}_{L^\infty(\R)} \int_{\epsilon<\abs{\theta}<2\epsilon} \abs{\theta} m(\abs{\theta})^{k-j+1} \dd{\theta} 
    \\[2pt]
    & \lesssim
    \frac{k!m(\epsilon)^{k+1}}{n^k} + 
    \sum_{j=1}^k \frac{k!}{j!} \frac{1}{n^k} \frac{\norm{\phi_\ell^{(j)}}_{L^\infty(\R)}}{\epsilon^{j-2}} m(\epsilon)^{k-j+1} 
    \\[2pt]
    & \lesssim
    \left( 1 + \epsilon + \frac{1}{2} \right) \frac{k!m(\epsilon)^{k+1}}{n^k} + 
    \sum_{j=3}^k \frac{k!}{j!} \frac{1}{n^k} \left( \frac{8\ell}{\epsilon} \right)^{j-2} m(\epsilon)^{k-j+1} 
    \\[2pt]
    & \lesssim
    \frac{k!m(\epsilon)^{k+1}}{n^k} + 
    \frac{m(\epsilon)}{n^2} \sum_{j=3}^k \frac{k!}{j!} \left( \frac{8\ell}{n\epsilon} \right)^{j-2} \left( \frac{m(\epsilon)}{n} \right)^{k-j}.
\end{align*}
Let $c \in (0,1)$. We now choose $\ell = k = k(\epsilon,n) = \lfloor cn/m(\epsilon) \rfloor$. Then, using Stirling's formula to observe that ${k! \lesssim (k/ce)^k}$ for $k\ge1$, we have
\begin{align*}
    \norm{x^{\epsilon,\ell}_n}
    & \lesssim
    m(\epsilon) \left( \frac{km(\epsilon)}{cen} \right)^k + 
    \frac{m(\epsilon)}{n^2} \sum_{j=3}^k \frac{k!}{j!} \left( \frac{8k}{n\epsilon} \right)^{j-2} \left( \frac{m(\epsilon)}{n} \right)^{k-j}
    \\[2pt]
    & \lesssim
    m(\epsilon) \exp\left( -\frac{cn}{m(\epsilon)} \right) + 
    \frac{m(\epsilon)}{n^2} \sum_{j=3}^k \frac{k!}{j!} \left( \frac{8c}{\epsilon m(\epsilon)} \right)^{j-2} \left( \frac{c}{k} \right)^{k-j}
    \\[2pt]
    & \lesssim
    m(\epsilon) \exp\left( -\frac{cn}{m(\epsilon)} \right) + 
    \frac{m(\epsilon)}{n^2} c^{k} \sum_{j=3}^k \left( \frac{8}{\epsilon m(\epsilon)} \right)^{j-2}.
\end{align*}
We now take $\epsilon= \epsilon(n) = m_{\log}^{-1}(cn)$. Then
\begin{equation*}
    m(\epsilon) \exp\left( -\frac{cn}{m(\epsilon)} \right)
    = 
    m(\epsilon) \exp\left( -\log\left( 1 + \frac{m(\epsilon)}{\epsilon} \right) \right)
    \leq
    \epsilon,
\end{equation*}
and
\begin{equation*}
    \frac{m(\epsilon)}{n^2} \lesssim \frac{1}{n}.
\end{equation*}
Hence
\begin{equation}\label{eq:xne bound}
    \norm{x^{\epsilon,\ell}_n} \lesssim \epsilon + \frac{c^{k}}{n} \sum_{j=3}^k \left( \frac{8}{\epsilon m(\epsilon)} \right)^{j-2}.
\end{equation}
Let $\delta\coloneqq \liminf_{\epsilon\to0+} \epsilon \,m(\epsilon)$. We distinguish several cases. Suppose first that $\delta >0$. Then $1/n = O(m_{\log}^{-1}(cn))$, and if $\delta > 8$, then \eqref{eq:mlog seq bound} follows from \eqref{eq:x_n-x_ne bound} and \eqref{eq:xne bound}. On the other hand, if $\delta \in(0,8]$ then we define $\widetilde{m}\colon (0,\pi] \to (0,\infty)$ by $\widetilde{m}(\epsilon) \coloneqq m(\delta\epsilon/9)$. We see that $\widetilde{m}$ is non-increasing with $\widetilde{m}(\epsilon) \geq m(\epsilon)$ and $\liminf_{\epsilon\to0+} \epsilon\, \widetilde{m}(\epsilon) = 9$. So we may replace $m$ by $\widetilde{m}$ in \eqref{eq:bdy func deriv bound}, and the previous estimates give $\norm{x_n} = O(\widetilde{m}_{\log}^{-1}(cn))$ as $n \to \infty$. Now, note that
\begin{equation*}
    \widetilde{m}_{\log}(\epsilon)
     =
    m(\delta\epsilon/9) \log\left( 1 + \frac{m\left( \delta\epsilon/9  \right)}{\epsilon} \right) 
    <
    m_{\log}(\delta\epsilon/9).
\end{equation*}
Hence $\widetilde{m}_{\log}^{-1}(cn) \lesssim m_{\log}^{-1}(cn)$, so we obtain \eqref{eq:mlog seq bound} whenever $\delta > 0$. Finally, suppose that $\delta = 0$.  Expanding $F_x$ as a Taylor series and using \eqref{eq:bdy func deriv bound} shows that $F_x$ extends holomorphically to the domain
\begin{equation*}
    \Omega_m \coloneqq \left\{\lambda \in \C : \abs{\lambda - e^{i\theta}} < \frac{1}{m(\abs{\theta})} \text{ for some $\theta\in[-\pi,\pi]\setminus\{0\}$} \right\}.
\end{equation*}
By composing with the Cayley transform and using the `edge-of-the-wedge theorem' (see \cite[Theorem~B]{Rud71_Edge_Wedge}), we deduce that $G_x$ extends holomorphically to $\E \cup \Omega_m$. The assumption that $\delta = 0$ implies that $1 \in \Omega_m$ and that $G_x$ has a holomorphic extension to a domain containing the entire unit circle $\T$. Let 
\begin{equation*}
    r_0 \coloneqq \inf\big\{ r \in (0,1] : \text{$G_x$ extends holomorphically to $\C \setminus r\Dbar$}\big\} .
\end{equation*}
Then, $r_0\in[0,1)$ and we may apply Cauchy's integral formula to the extension of $G_x$ to deduce that $\norm{x_n} = O(r^n)$ as $ n \to \infty$ for any $r>r_0$.
 We may find  $\theta_{0} \in [-\pi,\pi] \setminus \{0\}$ such that $m(\abs{\theta_{0}}) \geq c(1-r_0)^{-1}$. To see this, suppose instead that $m(\epsilon) < c(1-r_0)^{-1}$ for all $\epsilon \in (0,\pi]$. Then, using \eqref{eq:bdy func deriv bound}, we see that for every $\theta \in (-\pi,\pi] \setminus \{0\}$ the power series
\begin{equation*}
	\sum_{j=0}^\infty \frac{F_x^{(j)}(e^{i\theta}) }{j!}(\lambda - e^{i\theta})^j
\end{equation*}
converges absolutely whenever $\abs{\lambda - e^{i\theta}} < c^{-1}(1-r_0)$, and hence $G_x$ has a holomorphic extension to the set $\C \setminus (1 - c^{-1}(1-r_0)) \Dbar$. But $1 - c^{-1}(1-r_0) < r_0$,  contradicting the definition of $r_0$.
Hence, taking $\theta_0$ as above and using the fact that $m$ is non-increasing, we see that for all $\epsilon \in (0, \abs{\theta_0}]$,
\begin{equation*}
    m_{\log}(\epsilon) 
    \geq m(\abs{\theta_0}) \log\left( 1 + \frac{m(\abs{\theta_0})}{\epsilon} \right)
    \geq \frac{c}{1-r_0} \log\left( 1 + \frac{c}{(1-r_0) \epsilon} \right).
\end{equation*}
Taking $\epsilon = m_{\log}^{-1}(cn)$ for $n \in \N$ sufficiently large gives
\begin{equation*}
    cn \geq \frac{c}{1-r_0} \log\left( 1 + \frac{c}{(1-r_0) m_{\log}^{-1}(cn)} \right),
\end{equation*}
which rearranges to give
\begin{equation*}
    m_{\log}^{-1}(cn) \geq \frac{c}{(1-r_0)(e^{(1-r_0)n} - 1)}.
\end{equation*}
Since $r_0 \in [0,1)$ there exists $r\in(r_0, e^{-(1-r_0)})$, and we obtain
\begin{equation*}
    \norm{x_n} = O(r^n) = O(e^{-(1-r_0)n}) = O\big(m_{\log}^{-1}(cn)\big), \qquad n \to \infty,
\end{equation*}
as required. 
\end{proof}

As noted in \cite{Sei15_Tauberian}, we may obtain a quantified version of the Katznelson--Tzafriri theorem as a corollary to the above Tauberian theorem. Given a complex Banach space $X$ and a power-bounded operator $T \in \B(X)$, we consider the bounded sequence $x = (T^n(I-T))_{n\ge0}$  in $\B(X)$ and note that, for $\lambda \in \E$,
\begin{equation*}
	G_x(\lambda) = (I-T) R(\lambda,T) = I + (1-\lambda)R(\lambda,T).
\end{equation*} 
Applying Theorem~\ref{thm:disct taub mod} gives the following.

\begin{corollary}\label{cor:quantified KT}
Let $X$ be a complex Banach space and let $T \in \B(X)$ be a power-bounded operator with $\sigma(T) \cap \T \subseteq \{1\}$. If $m\colon (0,\pi] \to (0,\infty)$ is a continuous, non-increasing function such that
\begin{equation*}
    \norm{R(e^{i\theta}, T)} \leq m(\abs{\theta}), \qquad \abs{\theta} \in (0,\pi],
\end{equation*}
then, for every $c \in (0,1)$,
\begin{equation*}
    \norm{T^n(I-T)} = O\big(m_{\log}^{-1}(cn)\big), \qquad n \to \infty.
\end{equation*}
\end{corollary}

Suppose $T$ and $m$ are as in Corollary~\ref{cor:quantified KT} and that $m(\epsilon) \lesssim \epsilon^{-\alpha}$ for some $\alpha \geq 1$. Then the conclusion becomes
\begin{equation}\label{eq:KT poly rate}
    \norm{T^n(I-T)} = O\bigg( \left( \frac{\log n}{n} \right)^{1/\alpha} \bigg), \qquad n \to \infty.
\end{equation}
It is further shown in \cite[Corollary~3.1]{Sei16_KT_Rates} that if $\liminf_{\abs{\theta} \to 0+} \abs{\theta}^{\alpha} \norm{R(e^{i\theta},T)} > 0$, then we obtain the lower bound
\begin{equation}\label{eq:KT lower}
    \norm{T^n(I-T)} \gtrsim n^{-1/\alpha}.
\end{equation}
One may ask whether it is always possible to remove the factor of $\log n$ in the estimate \eqref{eq:KT poly rate} and attain the lower bound \eqref{eq:KT lower}. This was answered in the negative in \cite[Theorem~3.6]{Sei16_KT_Rates} for $\alpha >2$, where a Banach space $X_\alpha$ and a power-bounded operator $T_\alpha \in \B(X_\alpha)$ were constructed such that $\sigma(T_\alpha)\cap\mathbb{T}\subseteq\{1\}$, $\norm{R(e^{i\theta},T)} = O(\abs{\theta}^{-\alpha})$ as $\abs{\theta} \to 0+$, and 
\begin{equation*}
    \limsup_{n\to\infty} \left( \frac{n}{\log n} \right)^{1/\alpha} \norm{T_\alpha^n (I - T_\alpha)} > 0.
\end{equation*}
Nevertheless, there are some situations in which the logarithmic factor may be removed and the lower bound in \eqref{eq:KT lower} is attained. This is the case when $X$ is a Hilbert space and the function $m$ is of the form $m(\epsilon)=C\epsilon^{-\alpha}$ for some $C>0$ and $\alpha\ge1$ (see \cite[Theorem~3.10]{Sei16_KT_Rates}), or more generally when $m$ has `reciprocally positive increase' (see  \cite[Theorem~2.3]{NgSei20_Optimal_KT}). The factor $\log n$ may also be removed for general Banach spaces when $\alpha = 1$. It is this class of operators we turn to now.  

\begin{definition}
Let $X$ be a complex Banach space. We say that $T \in \B(X)$ is a \emph{Ritt operator}  if $\sigma(T) \subseteq \Dbar$ and there exists $C > 0$ such that
\begin{equation}\label{eq:Ritt bound}
    \norm{R(\lambda,T)} \leq \frac{C}{\abs{\lambda - 1}},
    \qquad\lambda\in\E.
\end{equation}
\end{definition}

It was shown in \cite{Kom68_Ergodic} that if $T$ is a Ritt operator, then $T$ is power-bounded and satisfies
\begin{equation}\label{eq:operator Ritt rate}
    \|T^n(I-T)\|=O(n^{-1}),\qquad  n \to \infty.
\end{equation}
The converse to this statement was shown in \cite[Theorem~2.1]{Nev94_Power_bounded_resolvent}; see also \cite{Lyu99_Spectral_localization} and \cite{NagZem99_Resolvent_Condition}.

In fact, the implication that a power-bounded operator satisfying \eqref{eq:Ritt bound} must also satisfy \eqref{eq:operator Ritt rate} can be shown using the same approach as in Theorem~\ref{thm:disct taub mod} and Corollary~\ref{cor:quantified KT}. It is a consequence of the following Tauberian theorem.

\begin{proposition}\label{prop:Ritt sequences}
Let $X$ be a complex Banach space and suppose that $x \in \ell^\infty(\Z_+;X)$ satisfies \eqref{eq:partial sum bound}. Suppose further that $G_x$ admits a boundary function $F_x \in \linebreak {C^2(\T\setminus\{1\}; X)}$ and that there exists $C > 0$ such that 
\begin{equation}\label{eq:bdy func Ritt bound}
    \norm{F_x^{(j)}(\theta)} \leq \frac{C}{\abs{\theta}^{j}},\qquad \abs{\theta} \in (0,\pi],
\end{equation}
for  all $j \in \{0,1,2\}$. Then $\|x_n\|=O(n^{-1})$ as $n\to\infty$.
\end{proposition}

\begin{proof}
Choose $\phi \in C^\infty(\R)$ such that $\phi(\theta) = 1$ for $\abs{\theta} < 1$ and $\phi(\theta) = 0$ and $\abs{\theta} > 2$. Similarly as in the proof of Theorem~\ref{thm:disct taub mod}, for $\epsilon \in (0,\pi/2]$ and $n \ge0$, we define $y_n^\epsilon \in X$ as
\begin{equation*}
    y^\epsilon_n \coloneqq \frac{1}{2\pi} \int_{-\pi}^\pi e^{in\theta} \psi(\theta/\epsilon) \dd{\theta},
\end{equation*}
where $\psi(\theta) \coloneqq 1 - \phi(\theta)$. Then $y^\epsilon \in \ell^1(\Z)$ and we define $x^\epsilon \coloneqq x * y^\epsilon$, where we have extended $x$ to a sequence in $\ell^\infty(\Z;X)$ by setting $x_n = 0$ for $n < 0$. As before, we now have $    \sup_{n \ge0} \norm{x_n - x^\epsilon_n} \lesssim \epsilon$, 
and 
\begin{equation*}
    x^\epsilon_n = \frac{1}{2\pi} \int_{-\pi}^\pi e^{i(n+1) \theta} F_x(e^{i\theta}) \psi(\theta/\epsilon) \dd{\theta}.
\end{equation*}
Integrating by parts twice gives
\begin{align*}
    x^\epsilon_n 
    &= 
    \frac{1}{2\pi}\left[ 
    \frac{1}{(n+1)(n+2)} \int_{-\pi}^\pi e^{i(n+3)\theta} F_x''(e^{i\theta}) \psi(\theta/\epsilon) \dd{\theta}
    \right.
    \\
    & \qquad - 
    \frac{i}{\epsilon}\left( \frac{1}{(n+1)^2} + \frac{1}{(n+1)(n+2)} \right) \int_{-\pi}^\pi e^{i(n+2)\theta} F_x'(e^{i\theta}) \psi'(\theta/\epsilon) \dd{\theta} 
    \\
    & \qquad -
    \left.
    \frac{1}{\epsilon^2(n+1)^2} \int_{-\pi}^\pi e^{i(n+1)\theta} F_x(e^{i\theta}) \psi''(\theta/\epsilon) \dd{\theta}
    \right],
\end{align*}
and applying \eqref{eq:bdy func Ritt bound} gives
\begin{align*}
    \norm{x^\epsilon_n} 
    & \lesssim
    \frac{1}{n^2} \left[
    \int_{\epsilon < \abs{\theta} < \pi} \abs{\theta}^{-2} \dd{\theta}
    + \frac{1}{\epsilon} \int_{\epsilon < \abs{\theta} < 2\epsilon} \abs{\theta}^{-1} \dd{\theta}
    +
    \frac{1}{\epsilon^2} \int_{\epsilon < \abs{\theta} < 2\epsilon} \dd{\theta}
    \right]
    \lesssim
    \frac{1}{n^2 \epsilon}.
\end{align*}
The result follows on setting $\epsilon = n^{-1}$.
\end{proof}

Since Ritt operators are power-bounded, we recover the following result.

\begin{corollary}
Let $X$ be a complex Banach space and let $T \in \B(X)$ be a Ritt operator. Then 
$    \norm{T^n(I-T)} = O(n^{-1})$ as $ n \to \infty$.
\end{corollary}


\section{Tauberian Theorems with Several Boundary Singularities}
\label{sec:several points}

The previous results have assumed that, given a complex Banach space $X$ and a bounded sequence $x \in \ell^\infty(\Z_+;X)$ satisfying \eqref{eq:partial sum bound}, the boundary function $F_x$ has at most one singularity on $\T$, specifically at the point $1$. In fact, the approach used in the proof of Theorem~\ref{thm:disct taub single} can be easily generalised to the case where, given a finite subset $E $ of $ \T$, we have $F_x \in L^1_{\loc}(\T \setminus E; X)$ or $F_x \in C^k(\T \setminus E; X)$ for some $k \in \N \cup \{\infty\}$. Given a finite set $E\subseteq\mathbb{T}$, we let
\begin{equation*}
    C_E(\T) \coloneqq \{ \psi \in C(\T) : \supp \psi \subseteq \T \setminus E\}.
\end{equation*}

\begin{definition}
Let $X$ be a complex Banach space, $G\colon \E \to X$ a holomorphic function and $E $ a finite subset of $ \T$. We say that $F \in L^1_{\loc}(\T \setminus E; X)$ is an \emph{$E$-boundary function} for $G$ if, for all $\psi \in C_E(\T)$,
\begin{equation*}
    \lim_{r \to 1+} \int_{-\pi}^\pi \psi(e^{i\theta}) G(re^{i\theta}) \dd\theta
    = \int_{-\pi}^\pi \psi(e^{i\theta}) F(e^{i\theta}) \dd\theta.
\end{equation*}
\end{definition}

Given $\ell, N \in \N$ and a finite set $E = \{e^{i\theta_1}, \ldots, e^{i\theta_N}\} \subseteq \T$, we can use Lemma~\ref{lem:step funcs conv} to construct functions $2\pi$-periodic functions $\phi_{\ell,E}$ with all of the following properties:
\begin{enumerate}[(i)]
    \item There exists $\delta > 0$ such that $\phi_{\ell,E}(\theta_p + 2k\pi - \theta) = 1$ whenever $\abs{\theta} \leq \delta$, $p \in \{1,\ldots,N\}$ and $k \in \Z$.
    
    \item There exists $\delta > 0$ such that $\phi_{\ell,E}(\theta_p + 2k\pi - \theta) = 0$ whenever $\abs{\theta} \geq 2\delta$, $p \in \{1,\ldots,N\}$ and $k \in \Z$.
    
    \item There exists $\delta > 0$ such that $\phi_{\ell,E}(\theta_p + 2k\pi - \theta) \in (0,1)$ whenever $\delta < \abs{\theta} < 2\delta$, $p \in \{1,\ldots,N\}$ and $k \in \Z$.
    
    \item There exists $C>0$ such that $\max_{1 \leq j \leq 2} \norm{\phi_{\ell,E}^{(j)}}_{L^\infty(\R)} \leq C$.
    
    \item There exists $C>0$ such that $\norm{\phi_{\ell,E}^{(j)}}_{L^\infty(\R)} \leq C (8 \ell)^{j-2}$ for $2 < j \leq \ell + 1$.
\end{enumerate}
The following generalisation of Theorem~\ref{thm:disct taub single} is obtained by replacing the bump functions $\phi_\ell \in C^{\ell+1}([-\pi,\pi])$ used in the proof of Theorem~\ref{thm:disct taub mod} with the functions $\phi_{\ell,E}$ restricted to the interval $[-\pi,\pi]$.

\begin{theorem}\label{thm:multi Tauberian}
Let $X$ be a complex Banach space, $N \in \N$ and $E = \{e^{i\theta_1}, \ldots, e^{i\theta_N}\}\subseteq\T$. Suppose that $x \in \ell^\infty(\Z_+; X)$ is such that
\begin{equation}\label{eq:multi partial sum bound}
    \sup_{n \geq 0} \, \bigg\|{\sum_{k=0}^n e^{-ik\theta_p} x_k}\bigg\| < \infty
\end{equation}
for all $p \in \{1, \ldots, N\}$.  If $G_x$ admits an $E$-boundary function $F_x \in L^1_{\loc}(\T \setminus E; X)$, then $x \in c_0(\Z_+; X)$.

Moreover, given $\delta > 0$ and a continuous non-increasing function $m\colon (0,\delta] \to (0,\infty)$, the following hold:
\begin{enumerate}
    \item[\textup{(i)}]\label{pt:finite deriv} Let $k \in \N$ and suppose that $F_x \in C^k(\T \setminus E; X)$ and that there exists  $C>0$ such that  
    \begin{equation*}
        \norm{F_x^{(j)}(e^{i(\theta - \theta_p)})} \leq C \abs{\theta}^{k - j} m(\abs{\theta})^{k+1},\qquad \abs{\theta} \in (0, \delta],
    \end{equation*}
for all $p \in \{1, \ldots, N\}$ and $0 \leq j \leq k$.    Then, for every $c > 0$,
    \begin{equation*}
        \norm{x_n} = O\big(m_k^{-1}(cn)\big), \qquad n \to \infty.
    \end{equation*}
    
    \item[\textup{(ii)}]\label{pt:smooth} Suppose that $F_x \in C^\infty(\T \setminus E; X)$ and that there exists  $C>0$ such that
    \begin{equation*}
        \norm{F_x^{(j)}(e^{i(\theta-\theta_p)})} \leq C j! \abs{\theta} m(\abs{\theta})^{j+1},\qquad \abs{\theta} \in (0, \delta],
    \end{equation*}
    for all $p \in \{1, \ldots, N\}$ and $ j \geq 0$.
    Then, for every $c \in (0,1)$,
    \begin{equation*}
        \norm{x_n} = O\big(m_{\log}^{-1}(cn)\big), \qquad n \to \infty.
    \end{equation*}
\end{enumerate}
\end{theorem}

Similarly to Corollary~\ref{cor:quantified KT}, applying this Tauberian theorem to the sequence $(T^n \prod_{\xi\in E} (\xi-T))_{n\ge0}$ where $T$ is a power-bounded operator gives the following generalisation of the quantified Katznelson--Tzafriri theorem.

\begin{corollary}
Let $X$ be a complex Banach space, $N \in \N$, $E = \{e^{i\theta_1}, \ldots, e^{i\theta_N}\}\subseteq\T$, and let $T \in \B(X)$ be a power-bounded operator with $\sigma(T) \cap \T \subseteq E$. If $\delta > 0$ and $m\colon (0,\delta] \to (0,\infty)$ is a continuous, non-increasing function such that
\begin{equation*}
    \norm{R(e^{i(\theta - \theta_p)}, T)} \leq m(\abs{\theta}), \qquad \abs{\theta} \in (0,\delta], 
\end{equation*}
for all $p \in \{1, \ldots, N\}$, then, for every $c \in (0,1)$,
\begin{equation*}
    \bigg\|{T^n \prod_{p=1}^N (e^{i\theta_p} - T)}\bigg\| = O\big(m_{\log}^{-1}(cn)\big), \qquad n \to \infty.
\end{equation*}
\end{corollary}

We now turn to operators satisfying a generalisation of the Ritt condition~\eqref{eq:Ritt bound}.

\begin{definition}[{\cite[Definition 2.1]{BouMer22}}]\label{def:ritt defns}
Let $X$ be a complex Banach space, $T \in \B(X)$ and $E$ a finite subset of $ \T$. We say $T$ is an \emph{$E$-Ritt operator} if $\sigma(T) \subseteq \Dbar$ and there exists $C > 0$ such that
\begin{equation*}
    \norm{R(\lambda,T)}
    \leq 
    \max\left\{ \frac{C}{\abs{ \lambda-\xi}} : \xi \in E \right\},
    \qquad \lambda\in\E.
\end{equation*}
\end{definition}

It is shown in \cite[Theorem~2.9]{BouMer22} that $E$-Ritt operators have a classification analogous to that of the classical Ritt operators. Specifically, for a complex Banach space $X$ and a finite subset $E $ of $ \T$, the operator $T \in \B(X)$ is $E$-Ritt if and only if $T$ is power-bounded and 
\begin{equation}\label{eq:E-ritt rate}
    \bigg\|{T^n \prod_{\xi \in E} (\xi - T)}\bigg\| = O(n^{-1}), \qquad n \to \infty.
\end{equation}
Using the same approach as in Proposition~\ref{prop:Ritt sequences}, we obtain the following Tauberian theorem for sequences with a boundary function satisfying a condition similar to the resolvent condition for an $E$-Ritt operator.

\begin{proposition}
Let $X$ be a complex Banach space, $N \in \N$, $E = \{e^{i\theta_1}, \dots, e^{i\theta_N}\}\subseteq\T $, and suppose that $x \in \ell^\infty(\Z_+; X)$ satisfies \eqref{eq:multi partial sum bound}. Suppose further that $G_x$ admits an $E$-boundary function $F_x \in C^2(\T \setminus E; X)$, and that there exist $C,\delta > 0$ such that
\begin{equation*}
     \norm{F_x^{(j)}(e^{i(\theta - \theta_p)})} \le \frac{C}{\abs{\theta - \theta_p}^j},\qquad \abs{\theta} \in (0,\delta],
\end{equation*}
for all  $p \in \{1, \ldots, N\}$ and $j \in \{0,1,2\}$. Then $\|x_n\|=O(n^{-1})$ as $n\to\infty$.
\end{proposition}

Using this result and the fact that $E$-Ritt operators are power-bounded, we recover the decay rate in \eqref{eq:E-ritt rate}.

\begin{corollary}[{\cite[Theorem~2.9(ii)]{BouMer22}}]\label{thm:general multi ritt}
Let $X$ be a complex Banach space, $E $ a finite subset of $ \T$, and let $T \in \B(X)$ be an $E$-Ritt operator. Then
\begin{equation*}
    \bigg\|{T^n \prod_{\xi \in E} (\xi - T)} \bigg\|= O(n^{-1}),
    \qquad n \to \infty.
\end{equation*}
\end{corollary}

\end{document}